\documentclass{article}
\usepackage{amsmath}
\usepackage{amsthm,amssymb,enumerate,hyperref}
\usepackage[a4paper,margin=2cm,centering,nohead,ignorefoot,footskip=5ex]{geometry}
\usepackage{url}
\usepackage{xcolor}
\hypersetup{
    colorlinks,
    linkcolor={red!50!black},
    citecolor={red!50!black},
    urlcolor={blue!80!black}
}
\theoremstyle{plain}
\newtheorem{theorem}{Theorem}[section]

\newtheorem{notation}[theorem]{Notation}
\newtheorem{lemma}[theorem]{Lemma}
\theoremstyle{definition}
\newtheorem{remark}[theorem]{Remark}
\newtheorem{definition}[theorem]{Definition}
\newtheorem{corollary}[theorem]{Corollary}
\newtheorem{example}[theorem]{Example}

\newcommand{\N}{\mathbb{N}}

\newcommand{\norm}[1]{\left\lVert#1\right\rVert}

\newcommand{\Fix}{\mathrm{Fix}\,}
\newcommand{\proj}{\textup{proj}}
\newcommand{\Id}{\mathrm{Id}}

\begin{document}

\title{On the convergence of algorithms with Tikhonov regularization terms  }

\author{Bruno Dinis${}^{a}$ and Pedro Pinto${}^{b}$\\[2mm]
	\footnotesize ${}^{a}$ Departamento de Matem\'atica, Faculdade de Ci\^encias da
	Universidade de Lisboa,\\ 
	\footnotesize Campo Grande, Edif\'icio~C6, 1749-016~Lisboa, Portugal\\
	\footnotesize E-mail:  \protect\url{bmdinis@fc.ul.pt}\\[ 2mm]
	\footnotesize ${}^{b}$ Department of Mathematics, Technische Universit{\"a}t Darmstadt,\\ 
	\footnotesize Schlossgartenstrasse 7, 64289 Darmstadt, Germany \\
	\footnotesize E-mail:  \protect\url{pinto@mathematik.tu-darmstadt.de}
}
\maketitle

\begin{abstract}
We consider the strongly convergent modified versions of the Krasnosel'ski\u{\i}-Mann, the forward-backward and the Douglas-Rachford algorithms with Tikhonov regularization terms, introduced by Radu Bo\c{t}, Ern\"{o} Csetnek and Dennis Meier. We obtain quantitative information for these modified iterations, namely rates of asymptotic regularity and metastability. Furthermore, our arguments avoid the use of sequential weak compactness and use only a weak form of the projection argument.
\end{abstract}

\section{Introduction}

In nonlinear analysis one is often confronted with the need to find zeros of a sum of monotone operators. Two well-known iterative splitting methods to approximate such zeros are the forward-backward and the Douglas-Rachford algorithms (see \cite{BC(11)}). The former, weakly converges to a zero of the sum of a multi-valued maximal monotone operator with a single-valued cocoercive operator, while the latter weakly converges to a zero of the sum of two maximal monotone operators. These convergence results stem from the weak convergence of the Krasnosel'ski\u{\i}-Mann ($\mathsf{KM}$) iteration to a fixed point of a nonexpansive map $T$ in a Hilbert space $H$:
\begin{equation}\tag{$\mathsf{KM}$}
x_{n+1}=x_n+\lambda_n(T(x_n)-x_n),
\end{equation} 
with $x_0 \in H$  a starting point and $(\lambda_n) \subset [0,1]$ a sequence of real numbers. One way to obtain strong convergence is to impose stronger conditions on the operators, such as strong monotonicity or strong convexity. However, the application to certain cases may be unfeasible, as the stronger conditions may prove to be too restrictive.  
The proximal-Tikhonov algorithm (see \cite{LM(96)}) was introduced as an alternative way to obtain, under mild conditions, strong convergence to a zero of a maximal monotone operator $\mathsf{A}$. This algorithm is based on the proximal point algorithm \cite{Martinet,Rockafellar76} and, in a first step, switches from the operator $\mathsf{A}$ to the operator $\mathsf{A}+\mu_n \Id$, where $(\mu_n)$ is a sequence of nonnegative real numbers. As the original proximal point algorithm, in general, is only weakly convergent, the added term $\mu_n \Id$ is crucial in guaranteeing the strong convergence of the iteration.   
Motivated by this method, in \cite{Botetal} Radu Bo\c{t}, Ern\"{o} Csetnek and Dennis Meier considered the ($\mathsf{KM}$) iteration with Tikhonov regularization terms as follows:
\begin{equation}\tag{$\mathsf{T\textup{-}KM}$}\label{mKM}
x_{n+1}=\beta_n x_n+\lambda_n(T(\beta_nx_n)-\beta_nx_n),
\end{equation}
where $x_0 \in H$  is the starting point, $(\lambda_n), (\beta_n)$ are sequences of positive numbers and $T:H \to H$ is a nonexpansive mapping. Bo\c{t} \emph{et al.} obtained a strong convergence result for \eqref{mKM} and for similar modifications of the forward-backward and the Douglas-Rachford algorithms.

In this paper we obtain rates of asymptotic regularity and metastability for these modified iterations through a quantitative analysis of the proofs in \cite{Botetal}. As such we are able to obtain quantitative versions of the convergence results by Bo\c{t} \emph{et al.}. Similarly to \cite{FFLLPP(19),DP(21),DP(ta),PP(ta)}, our analysis bypasses the use of sequential weak compactness and only relies on a weak form of the metric projection argument (as explained in \cite{K(11),FFLLPP(19)}). Our results are guided by proof theoretical techniques (in the context of the proof mining program \cite{K(08),K(17)}), however no special knowledge of mathematical logic is required to read this paper.

\section{Preliminaries}\label{sectionPrelim}

Consider $H$ a real Hilbert space with inner product $\langle \cdot, \cdot \rangle$ and norm $\norm{\cdot}$. We start by recalling some notions and properties concerning operators on Hilbert spaces. An operator $\mathsf{A}:H \rightrightarrows H$  is said to be \emph{monotone} if and only if whenever $(x,y)$ and $(x',y')$ are elements of the graph of $\mathsf{A}$, it holds that $\langle x-x',y-y'\rangle \geq 0$. A monotone operator $\mathsf{A}$ is said to be \emph{maximal monotone} if  the graph of $\mathsf{A}$ is not properly contained in the graph of any other monotone operator on $H$. We denote by $zer(\mathsf{A})$, the set of all zeros of $\mathsf{A}$. 
Let $\mathsf{A}:H \rightrightarrows H$ be a maximal monotone operator and $\gamma >0 $. The resolvent function $J_{\gamma \mathsf{A}}$ is the single-valued function defined by $J_{\gamma \mathsf{A}} = (I + \gamma\mathsf{A} )^{-1}$, and the reflected resolvent function is the function $R_{\gamma \mathsf{A}}:=2 J_{\gamma \mathsf{A}}-\Id$.

A mapping $T : H \to H$ is called \emph{nonexpansive} if $\norm{T(x) -T(y)}\leq \norm{x -y},$ for all $x, y \in H$.  
The set 
of fixed points of the mapping $T$ will be denoted by  $\Fix T$.  If $T$ is nonexpansive, then $\Fix T$ is a closed and convex subset of $H$. 

For $\alpha \in (0,1]$, a functional $T:H \to H$ is called \emph{$\alpha$-averaged}\footnote{The standard definition asks for $\alpha \in (0,1)$. With this extension, $1$-averaged is just another way of saying nonexpansive.} if there exists a nonexpansive operator $T':H \to H$ such that $T=(1-\alpha)\Id+\alpha T'$. The $\alpha$-averaged operators are always nonexpansive. The $\frac{1}{2}$-averaged operators are also called firmly nonexpansive operators.  The resolvent function $J_{\gamma \mathsf{A}}$ is firmly nonexpansive and the reflected resolvent function $R_{\gamma \mathsf{A}}$ is nonexpansive.

For $\delta >0$, a functional $\mathsf{B}:H \to H$ is said to be \emph{$\delta$-cocoercive} if for all $x,y \in H$ it holds that $\langle x-y, \mathsf{B}(x)-\mathsf{B}(y)\rangle \geq \gamma \norm{\mathsf{B}(x)-\mathsf{B}(y)}^2$.

For a comprehensive introduction to convex analysis and the theory of monotone operators in Hilbert spaces we refer to \cite{BC(11)}.

The main purpose of this paper is to extract quantitative information from the proof of the following result. 
\begin{theorem}{\rm (\cite[Theorem~3]{Botetal})}\label{t:original}
Let $(\beta_n), (\lambda_n) \subset (0,1]$ be real sequences satisfying:
\[
\begin{array}{lll}
(i) \lim \beta_n =1, & (ii)\sum_{n \geq 0} (1-\beta_n)= \infty, & (iii)\sum_{n \geq 1} |\beta_n - \beta_{n-1}|< \infty,\\
 (iv) \liminf \lambda_n >0, &(v) \sum_{n \geq 1} |\lambda_n-\lambda_{n-1}|< \infty.&
\end{array} 
\]
Consider the iterative scheme \eqref{mKM} with and arbitrary starting point $x_0 \in H$ and a nonexpansive
mapping $T : H \to H$ such that  $\Fix T \neq \emptyset$ . Then $(x_n)$ converges strongly to $\proj_{\Fix T} (0)$.
\end{theorem}

Our main result (Theorem~\ref{t:main}) gives a bound on the \emph{metastability} property of the sequence $(x_n)$, i.e.
\begin{equation*}
\forall k \in\N \, \forall f:\N \to \N \, \exists n \in \N \, \forall i,j \in [n,f(n)] \,(\norm{x_i-x_j}\leq \frac{1}{k+1} ),
\end{equation*}
which is (non-effectively) equivalent to the Cauchy property for  $(x_n)$.

We now present some useful quantitative lemmas which require the notion of monotone functional for two particular cases. This relies on  the strong majorizability relation from \cite{bezem1985strongly}. 
\begin{definition}
For functions $f,g:\N\to\N$ we define
$$g\leq^* f := \forall n, m\in\N\,\left(m\leq n\to \left( g(m)\leq f(n) \land f(m)\leq f(n)\right)\right).$$
A function $f:\N\to\N$ is said to be \emph{monotone} if  $f\leq^* f$.
We say that a functional $\varphi:\N\times \N^{\N}\to\N$ is \emph{monotone} if for all $m,n\in\N$ and all $f,g:\N\to\N$,
$$\left(m\leq n \land g\leq^* f\right) \to \left(\varphi(m,g)\leq \varphi(n,f)\right).$$
\end{definition}

\begin{remark}\label{maj}
For $f:\N \to \N$, the notion of being monotone corresponds to saying that $f$ is a  nondecreasing function, i.e.\ $\forall n\in\N\, \left(f(n)\leq f(n+1)\right)$. 

In what follows we restrict our metastability results to monotone functions in $\N^\N$. However, there is no real restriction in doing so, as for $f: \N \to \N$, one has $f \leq^* \! f^{\mathrm{maj}}$, where $f^{\mathrm{maj}}$ is the monotone function defined by $f^{\mathrm{maj}}(n):= \max\{f(i)\, :\, i \leq n\}$. In this way, we avoid constantly having to switch from $f$ to  $f^{\mathrm{maj}}$, and simplify the notation.
\end{remark}

\begin{notation}\label{notation}
	Consider a function $\varphi$ on tuples of variables $\bar{x}$, $\bar{y}$. If we wish to consider the variables $\bar{x}$ as parameters we write $\varphi[\bar{x}](\bar{y})$. For simplicity of notation we may then even omit the parameters and simply write $\varphi(\bar{y})$.
\end{notation}

\begin{lemma}[\cite{LLPP(ta)}]\label{LemmaLLPP}
Let $(s_n)$ be a bounded sequence of non-negative real numbers, with $d\in\N \setminus\{0\}$ an upper bound for $(s_n)$, such that for any $n\in \N$
	\begin{equation*}
		s_{n+1}\leq (1-\alpha_n)s_n+\alpha_nr_n + \gamma_n,
	\end{equation*}
	\noindent where $(\alpha_n)\subset [0,1]$, $(r_n)$ and $(\gamma_n)\subset [0,+\infty)$ are given sequences of real numbers.\\	
	Assume that exist functions $A$, $R$, $G:\N \to \N$ such that, for all $k \in \N$
	\[
\begin{array}{lll}
		 (i)& \sum\limits_{i=1}^{A(k)} \alpha_i \geq k\, ;  &(ii)\, \forall n\geq R(k) \, \left( r_n \leq \dfrac{1}{k+1}\right);\\
		 (iii)&\,  \forall n \in \N \, \left( \sum\limits_{i=G(k)+1}^{G(k)+n} \gamma_i \leq \dfrac{1}{k+1} \right).&
	\end{array}
	\]
	Then
	\begin{equation*}
		\forall k \in \N \,\forall n\geq \theta(k) \, (s_n\leq \dfrac{1}{k+1}),
	\end{equation*}
	\noindent with $\theta(k):=\theta[A, R, G, d](k):=A(M-1+\lceil \ln(3d(k+1))\rceil)+1$, where \\$M:=\max\{ R(3k+2), G(3k+2)+1 \}$.
	\end{lemma}
	
The next result is an easy adaptation of \cite[Lemma~14]{PP(ta)} for the case where $(\gamma_n)\equiv 0$. 
\begin{lemma}\label{l:Xu2}
Let $(s_n)$ be a bounded sequence of non-negative real numbers and $d\in\N$ a positive upper bound on $(s_n)$. Consider sequences of real numbers $(\alpha_n)\subset (0,1)$, $(r_n)$ and $(v_n)$ and assume the existence of a monotone function $A$ such that $\sum_{i=1}^{A(k)} \alpha_i\geq k$, for all $k \in \N$. For natural numbers $k, n$ and $q$ assume
\[\forall i\in[n,q]\, (v_i\leq \frac{1}{3(k+1)(q+1)}\land r_i\leq \frac{1}{3(k+1)}),\]
and $s_{i+1}\leq (1-\alpha_i)(s_i+v_i)+\alpha_ir_i$, for all $i\in\N$.
Then
\[\forall i\in[\sigma(k,n),q]\, (s_i\leq \frac{1}{k+1}),\]
with $\sigma(k,n):=\sigma[A, d](k,n):=A\left(n+\lceil \ln(3d(k+1))\rceil\right)+1$.
\end{lemma}
The next result corresponds to a quantitative version of a weaker form of the projection argument of zero onto $\Fix T$. Below $B_{N}:=\{x \in H: \norm{x-p}\leq N\}$, where $p \in \Fix T$ is made clear by the context and $N \in \N$.
\begin{lemma}[\cite{PP(ta)}]\label{l:projection}
	 Let  $N\in \N\setminus \{0\}$ be such that $N\geq 2\norm{p}$ for some point $p\in \Fix T$.
	 For any $k\in \N$ and monotone function $f:\N \to \N $, there are $n \leq 24N(\check{f}^{(R)}(0)+1)^2$ and $x\in {\rm B}_N$ such that $\|T(x)-x \| \leq \frac{1}{f(n)+1}$ and 
	 $$ \forall y\in\! B_N 
	 (\|T(y)-y\|\leq \frac{1}{n+1} \to \langle x,x-y\rangle \leq \frac{1}{k+1}),$$ 
	 with $R:=4N^4(k+1)^2$ and $\check{f}:=\max\{ f(24N(m+1)^2),\, 24N(m+1)^2 \}$.	
\end{lemma}

\section{Quantitative results}\label{Sectionanalysis}

In this section we give a quantitative analysis of Theorem~\ref{t:original} as well as some of its corollaries. We start by stating the relevant quantitative conditions.

Given positive real sequences $(\beta_n), (\lambda_n) $, a constant $\ell \in \N\setminus \{0\}$ and functions $b,D,B,L: \N \to \N$ and $h: \N\to \N\setminus \{0\}$ we consider the following conditions: 
\begin{enumerate}[($Q_1$)]
\item\label{Q1} $\forall n \in \N(\beta_n \geq \frac{1}{h(n)})$
\item\label{Q2} $\forall k \in \N \,\forall n \geq b(k) (|1-\beta_n| \leq \frac{1}{k+1})$,
\item\label{Q3} $\forall k \in \N \left(\sum_{i=1}^{D(k)}(1- \beta_i) \geq k\right)$,
\item\label{Q4}  $\forall k \in \N \, \forall n \in \N \left(\sum_{i=B(k)+1}^{B(k)+n}|\beta_{i}-\beta_{i-1}|\leq \frac{1}{k+1} \right)$,
\item\label{Q5} $\forall n \in \N (\lambda_n \geq \frac{1}{\ell})$,
\item\label{Q6}  $\forall k \in \N \, \forall n \in \N \left(\sum_{i=L(k)+1}^{L(k)+n}|\lambda_{i}-\lambda_{i-1}|\leq \frac{1}{k+1} \right)$.
\end{enumerate}

The first condition states that the function $h$ witnesses the fact that $\beta_n  >0$. The conditions $(Q_{\ref{Q2}})-(Q_{\ref{Q6}})$ correspond to quantitative strenghtenings of the conditions $(i)-(v)$ in Theorem~\ref{t:original}. Indeed, $b$ is a rate of convergence for $\beta_n \to 1$; the function $D$ is a rate of divergence for $\left(\sum(1- \beta_i)\right)$; the functions $B$ and $L$ are rates for the Cauchy property of the convergent series $\sum|\beta_{i}-\beta_{i-1}|$ and $\sum|\lambda_{i}-\lambda_{i-1}|$, respectively, and the number $\ell$ is used to express the fact that $(\lambda_n)$ is bounded away from $0$.

\begin{example}
Consider the sequences defined by $\lambda_n=\beta_n=1-\frac{1}{n+2}$, for all $n \in \N$. Then the conditions $(Q_{\ref{Q1}})-(Q_{\ref{Q6}})$ are satisfied with $h \equiv 2$, $b=B=L=\Id$, $D(k):=e^{k+2}$ and $\ell=2$. 
\end{example}

In the following lemmas $(x_n)$ always denotes a sequence generated by \eqref{mKM}.

\begin{lemma}
Let $N\in \N \setminus\{0\}$ be such that $N \geq \max\{\norm{x_0-p},\norm{p}\}$, for some $p \in \Fix T$. Then, $\norm{x_n} \leq 2N$ and $\norm{T(x_n)} \leq 2N$, for all $n \in \N$.
\end{lemma}

\begin{proof}
Let $p \in \Fix T$ be such that $N \geq \max\{\norm{x_0-p},\norm{p}\}$. Since $T$ is nonexpansive, we have
\begin{equation*}
\begin{split}
\norm{x_{n+1}-p}& \leq (1-\lambda_n)\norm{\beta_nx_n-p}+\lambda_n\norm{T(\beta_nx_n)-T(p)}\\
& \leq \norm{\beta_nx_n-p} = \norm{\beta_n(x_n-p)+(\beta_n-1)p}\\
& \leq \beta_n \norm{x_n-p}+ (1-\beta_n)\norm{p}.
\end{split}
\end{equation*}
By induction on $n$ one easily shows that $\norm{x_n-p} \leq N$. Since $T$ is nonexpansive, $\norm{T(x_n)-p}\leq N$ and the result follows from the fact that $N \geq \norm{p}$.
\end{proof}
The next lemma gives a rate of asymptotic regularity for the sequence $(x_n)$.
\begin{lemma}\label{l:asymptotic_regularity}
Let $N \in \N \setminus\{0\}$ be such that $N \geq \max\{\norm{x_0-p},\norm{p}\}$, for some fixed point $p$. Consider $\ell \in \N \setminus\{0\}$ and monotone functions $b,D,B,L:\N \to \N$ satisfying conditions $(Q_{\ref{Q2}})-(Q_{\ref{Q6}})$. Then
\begin{enumerate}[$(i)$]
\item $\forall k \in \N\, \forall n \geq \nu_1(k) ( \norm{x_{n+1}-x_n}\leq \frac{1}{k+1})$,
\item $\forall k \in \N\, \forall n \geq \nu_2(k) ( \norm{T(x_{n})-x_n}\leq \frac{1}{k+1})$,
\end{enumerate}
$$\text{where }\quad\nu_1(k):=\nu_1[N,D,B,L](k):=\theta[D,\mathbf{0},G,2N](k)\text{, and}$$  
$$\nu_2(k):=\nu_2[N,\ell,b,D,B,L](k):=\max\{b(4N\ell(k+1)-1),\nu_{1}(2\ell(k+1)-1)\},$$
with $G(k):=G[N,B,L](k):=\max\{B(4N(k+1)-1),L(10N(k+1)-1)\}$, and $\theta$ is as in Lemma~\ref{LemmaLLPP} and $\mathbf{0}$ is the zero function.
\end{lemma}

\begin{proof}
For all $n \in \N$ we have 
\begin{equation}\label{e:Tbx}
\begin{split}
\norm{T(\beta_nx_n)}&\leq \norm{T(\beta_nx_n)-T(\beta_np)}+\norm{T(\beta_np)-T(p)}+\norm{T(p)}\\
& \leq \norm{x_n-p}+(1-\beta_n)\norm{p}+\norm{p} \leq 3N 
\end{split}
\end{equation}
Using \eqref{e:Tbx} we derive
\begin{equation*}
\begin{split}
\norm{x_{n+1}-x_n}&
 \leq \norm{(1-\lambda_n)(\beta_nx_n-\beta_{n-1}x_{n-1})+(\lambda_{n-1}-\lambda_n)\beta_{n-1}x_{n-1}}\\
& \quad+\norm{\lambda_n(T(\beta_{n}x_{n})-T(\beta_{n-1}x_{n-1}))+(\lambda_n-\lambda_{n-1})T(\beta_{n-1}x_{n-1})}\\
& \leq \norm{\beta_nx_n-\beta_{n-1}x_{n-1}}+ |\lambda_n-\lambda_{n-1}|(\norm{\beta_{n-1}x_{n-1}}+\norm{T(\beta_{n-1}x_{n-1})})\\
& \leq \norm{\beta_nx_n-\beta_{n-1}x_{n-1}}+ 5N|\lambda_n-\lambda_{n-1}|.
\end{split}
\end{equation*}
For $n \geq 1$, we have
\begin{equation*}
\begin{split}
\norm{x_{n+1}-x_n}&=\norm{\beta_n(x_n-x_{n-1})+(\beta_n-\beta_{n-1})x_{n-1}}+5N|\lambda_n-\lambda_{n-1}|\\
& \leq \beta_n\norm{x_n-x_{n-1}}+2N|\beta_n-\beta_{n-1}|+5N|\lambda_n-\lambda_{n-1}|.
\end{split}
\end{equation*}
Observe that the function $G$ is a Cauchy rate for $\left(\sum \gamma_n\right)$, where $(\gamma_n)$ is given by $\gamma_n:=2N|\beta_n-\beta_{n-1}|+5N|\lambda_n-\lambda_{n-1}|$. Indeed, for all $n,k$ we have 
\begin{equation*}
\begin{split}
\sum_{i=G(k)+1}^{G(k)+n} \gamma_i &= 2N \sum_{i=G(k)+1}^{G(k)+n} |\beta_i-\beta_{i-1}|+5N \sum_{i=G(k)+1}^{G(k)+n} |\lambda_i -\lambda_{i-1}|\\
& \leq 2N \sum_{i=B(4N(k+1)-1)+1}^{G(k)+n} |\beta_i-\beta_{i-1}|+5N \sum_{i=L(10N(k+1)-1)+1}^{G(k)+n} |\lambda_i -\lambda_{i-1}|\\
& \leq \frac{2N}{4N(k+1)} + \frac{5N}{10N(k+1)} =\frac{1}{k+1}.
\end{split}
\end{equation*}
Applying Lemma~\ref{LemmaLLPP} with $d=2N$, and for all $n \geq 1$, $s_n=\norm{x_n-x_{n-1}}$, $r_n= 0$, $\gamma_n=2N|\beta_n-\beta_{n-1}|+5N|\lambda_n-\lambda_{n-1}|$ and $\alpha_n=1-\beta_n$ we conclude Part $(i)$.
As for Part $(ii)$, observe that for all $n \in \N$ 
\begin{equation*}
\begin{split}
\norm{x_n -T(x_n)} &\leq \norm{x_{n+1}-x_n}+\norm{(1-\lambda_n)(\beta_nx_n-T(x_n))+\lambda_{n}(T(\beta_{n}x_{n})-T(x_n))}\\
& \leq \norm{x_{n+1}-x_n}+(1-\lambda_n)\norm{\beta_nx_n-T(x_n)}+\lambda_{n}\norm{\beta_{n}x_{n}-x_n}\\
& \leq \norm{x_{n+1}-x_n}+(1-\lambda_n)\norm{\beta_nx_n-\beta_nT(x_n)}\\
&\quad +(1-\lambda_n)\norm{\beta_{n}T(x_{n})-T(x_n)}+\lambda_n(1-\beta_n)\norm{x_n}\\
& \leq \norm{x_{n+1}-x_n}+(1-\lambda_n)\norm{x_n-T(x_n)}\\
& \quad+(1-\lambda_n)(1-\beta_n)\norm{T(x_{n})}+\lambda_n(1-\beta_n)\norm{x_n}.
\end{split}
\end{equation*}
Hence, for $n \geq \nu_2(k)$ 
\begin{equation*}
\begin{split}
\norm{x_n -T(x_n)} &\leq \frac{1}{\lambda_n}\left(\norm{x_{n+1}-x_n}+(1-\lambda_n)(1-\beta_n)\norm{T(x_{n})}+\lambda_n(1-\beta_n)\norm{x_n}\right)\\
& \leq \frac{1}{\lambda_n}\left(\norm{x_{n+1}-x_n}+2N(1-\lambda_n)(1-\beta_n)+2N\lambda_n(1-\beta_n)\right)\\
& \leq \ell \left(\norm{x_{n+1}-x_n}+ 2N(1-\beta_n)\right) \leq \frac{1}{k+1},
\end{split}
\end{equation*}
which shows Part $(ii)$.
\end{proof}

The next result provides quantitative information on the fact that, with $\tilde{x}=\proj_{\Fix T}(0)$, one has $\limsup \langle \tilde{x},\tilde{x}-x_n\rangle\leq 0$. Notice that, unlike the original proof, the quantitative version below does not require sequential weak compactness -- the elimination of this principle is justified in \cite{FFLLPP(19)}. 
\begin{lemma}\label{l:weak_compact}
 Let $N \in \N \setminus\{0\}$ be such that $N\geq \max\{\norm{x_0-p},\norm{p}\}$, for some fixed point $p$. Consider $\ell \in \N \setminus\{0\}$ and monotone functions $b,D,B,L:\N \to \N$ satisfying conditions $(Q_{\ref{Q2}})-(Q_{\ref{Q6}})$.
 For any $k\in \N$ and monotone function $f:\N \to \N $, there are $n \leq \psi(k,f)$ and $x\in B_{2N}$ such that
 $$\|T(x)-x \| \leq \frac{1}{f(n)+1} \, \land \, \forall m \geq n 
 \big( \langle x,x-x_m\rangle \leq \frac{1}{k+1}\big),$$ 
where $\psi(k,f):=\psi[N,\ell,b,D,B,L](k,f):=\nu_2(48N(\check{g}^{(R)}(0)+1)^2)$, $g(m):=f(\nu_2(m))$, $R:=64N^4(k+1)^2$ and $\check{g}:=\max\{ g(48N(m+1)^2),\, 48N(m+1)^2 \}$, with $\nu_2$ as in Lemma~\ref{l:asymptotic_regularity}.	
	 
\begin{proof}
Given $k \in \N$ and monotone $f:\N \to \N$, applying Lemma~\ref{l:projection} to $k$ and the function $g$ one obtains $n_0 \leq 48N(\check{g}^{(R)}(0)+1)^2$ and $x \in B_{2N}$ such that
$\norm{T(x)-x} \leq \frac{1}{g(n_0)+1}$ and for all $y \in B_{2N}$ 
\begin{equation}\label{e:weak_compactness}
	 \|T(y)-y\|\leq \frac{1}{n_0+1} \to \langle x,x-y\rangle \leq \frac{1}{k+1}.
\end{equation}
 Define $n:=\nu_2(n_0)$. By monotonicity, $n\leq \psi(k,f)$ and by the definition of $g$ we conclude that $\norm{T(x)-x}\leq \frac{1}{f(n)+1}$. By Part $(ii)$ of Lemma~\ref{l:asymptotic_regularity} we have $\norm{T(x_m)-x_m} \leq \frac{1}{n_0+1}$, for all $m \geq n$. The result follows from  \eqref{e:weak_compactness} since $(x_n) \subset B_{2N}$.
\end{proof}
\end{lemma}

We are now able to prove our main result.

\begin{theorem}\label{t:main}
Given $(\beta_n), (\lambda_n) \subset (0,1]$, $x_0 \in H$ and a nonexpansive mapping $T$, let $(x_n)$ be generated by \eqref{mKM}. Assume that there exist $\ell \in \N\setminus \{0\}$ and monotone functions $b, D,B,L: \N \to \N$ and $h: \N\to \N\setminus \{0\}$ such that the conditions $(Q_{\ref{Q1}})-(Q_{\ref{Q6}})$ hold. Assume that there exists $N \in \N \setminus\{0\}$ such that $N \geq \max\{\norm{x_0-p},\norm{p}\}$, for some $p \in \Fix T$.
Then for all $k \in \N$ and monotone function $f:\N\to \N$
\begin{equation*}
\exists n \leq \mu(k,f) \, \forall i,j \in [n,f(n)] (\norm{x_i-x_j}\leq \frac{1}{k+1} ),
\end{equation*}
where $\mu(k,f):=\mu[N,\ell,b,h,D,B,L](k,f):=\sigma(\widetilde{f},\max\{\psi(12\widetilde{k}+1)-1,\widetilde{f}),n_1\})$, with $\widetilde{k}:=4(k+1)^2-1$, $n_1:=b(54N^2(\widetilde{k}+1)-1)$, $\widetilde{f}(m):=\widetilde{f}[k,N,f,h](m):=3(10N+1)(\widetilde{k}+1)(\overline{f}(m)+1)h(\overline{f}(m))-1
$, for $\overline{f}(m):=\overline{f}[k,N,b,f,D](m):=f(\sigma(\widetilde{k},\max\{m,n_1\}))$, $\psi$ as in Lemma~\ref{l:weak_compact} and $\sigma:=\sigma[D,9N^2]$ as in Lemma~\ref{l:Xu2}.
\end{theorem}

\begin{proof}
Let $k \in \N$ and $f:\N \to \N$ monotone be given. With $\widetilde{k}:=4(k+1)^2-1$ and $n_1:=b(54N^2(\widetilde{k}+1)-1)$  we define the functions
\begin{equation*}
\overline{f}(m):=f(\sigma(\widetilde{k},\max\{m,n_1\}))
\end{equation*}
and 
\begin{equation*}
\widetilde{f}(m):=3(10N+1)(\widetilde{k}+1)(\overline{f}(m)+1)h(\overline{f}(m))-1.
\end{equation*}
By Lemma~\ref{l:weak_compact} there exist $n_0 \leq \psi(12(\widetilde{k}+1)-1,\widetilde{f})$ and $x \in B_{2N}$ such that $\norm{T(x)-x}\leq \dfrac{1}{\widetilde{f}(n_0)+1}$ and 
\begin{equation}\label{e:apply_proj}
 \langle x,x-x_m\rangle \leq \frac{1}{12(\widetilde{k}+1)}, \text{for all } m \geq n_0.
\end{equation}
We have, for all $n \in \N$
\begin{equation*}
\begin{split}
\norm{x_{n+1}-x}
& \leq (1-\lambda_n)\norm{\beta_nx_n-x}+\lambda_{n}\norm{T(\beta_{n}x_{n})-T(x)}+\lambda_n\norm{T(x)-x}\\
& \leq \norm{\beta_nx_n-x}+\norm{T(x)-x}.
\end{split}
\end{equation*}
Let $w_n:=\norm{T(x)-x}(2\norm{\beta_nx_n-x}+\norm{T(x)-x})$. Using $(Q_{\ref{Q1}})$, for all $n \in \N$
\begin{equation*}
\begin{split}
\norm{x_{n+1}-x}^2& \leq (\norm{\beta_nx_n-x}+\norm{T(x)-x})^2\\
& =\norm{\beta_nx_n-x}^2+ w_n\\
& =\norm{\beta_n(x_n-x)+(\beta_n-1)x}^2+ w_n\\
& = \beta_n^2\norm{x_n-x}^2+2\beta_n(1-\beta_n)\langle x,x-x_n\rangle+(1-\beta_n)^2\norm{x}^2+ w_n\\
& \leq \beta_n\norm{x_n-x}^2+(1-\beta_n)(2\beta_n\langle x,x-x_n\rangle+(1-\beta_n)\norm{x}^2)+ w_n\\
& \leq \beta_n(\norm{x_n-x}^2+\frac{w_n}{\beta_n})+(1-\beta_n)(2\beta_n\langle x,x-x_n\rangle+(1-\beta_n)\norm{x}^2)\\
& \leq \beta_n(\norm{x_n-x}^2+h(n)w_n)+(1-\beta_n)(2\beta_n\langle x,x-x_n\rangle+(1-\beta_n)\norm{x}^2)
\end{split}
\end{equation*}
We will apply Lemma~\ref{l:Xu2} with $s_n:=\norm{x_n-x}^2$, $v_n:=h(n)w_n$, $\alpha_n:=1-\beta_n$, $r_n:=2\beta_n\langle x,x-x_n\rangle+(1-\beta_n)\norm{x}^2$, $A$ is instantiated with the function $D$ and $d:=9N^2$.  Let us see that the conditions of the lemma are satisfied.
By $(Q_{\ref{Q3}})$, it is clear that $A$ satisfies the required condition.
We have that $$s_n =\norm{x_n-x}^2 \leq (\norm{x_n-p}+\norm{x-p})^2\leq (N+2N)^2=d.$$
With $n:=\max\{n_0,n_1\}$ by \eqref{e:apply_proj} and $(Q_{\ref{Q2}})$, for all $m \geq n$, 
\begin{equation*}
r_m=2\beta_m\langle x,x-x_m\rangle+(1-\beta_m)\norm{x}^2 \leq \frac{2}{12(\widetilde{k}+1)}+ \frac{(\norm{x-p}+\norm{p})^2}{54N^2(\widetilde{k}+1)}\leq \frac{1}{3(\widetilde{k}+1)}.
\end{equation*} 
Observe that $2\norm{\beta_nx_n-x}+\norm{T(x)-x} \leq 2(\norm{x_n}+\norm{x-p}+\norm{p})+1\leq 10N+1.$
For $m \leq \overline{f}(n_0)=f(\sigma(\widetilde{k},n))$, using the fact that the function $h$ is monotone we have
\begin{equation*}
\begin{split}
v_m&=h(m)\norm{T(x)-x}(2\norm{\beta_mx_m-x}+\norm{T(x)-x}) \leq \frac{h(\overline{f}(n_0))(10N+1)}{\widetilde{f}(n_0)+1}\\
& =\frac{h(\overline{f}(n_0))(10N+1)}{3(10N+1)(\widetilde{k}+1)(\overline{f}(n_0)+1)h(\overline{f}(n_0))}= \frac{1}{3(\widetilde{k}+1)(f(\sigma(\widetilde{k},n))+1)}
\end{split}
\end{equation*}
This shows that  $v_m\leq \frac{1}{3(\widetilde{k}+1)(f(\sigma(\widetilde{k},n))+1)}
$ and $r_m \leq \frac{1}{3(\widetilde{k}+1)}$, for all $m \in [n,f(\sigma(\widetilde{k},n))]$. 
 Hence, by Lemma~\ref{l:Xu2}
\begin{equation}\label{e:quasi_convergence}
\forall m \in [\sigma(\widetilde{k},n),f(\sigma(\widetilde{k},n))] \left(\norm{x_m-x}^2\leq \frac{1}{\widetilde{k}+1}\right).
\end{equation} 
We conclude that for $i,j \in [\sigma(\widetilde{k},n),f(\sigma(\widetilde{k},n))]$
\begin{equation*}
\norm{x_i-x_j}\leq \norm{x_i-x}+\norm{x_j-x}\leq \frac{1}{2(k+1)}+\frac{1}{2(k+1)}=\frac{1}{k+1},
\end{equation*}
which entails the result since $\sigma(\widetilde{k},n)\leq \mu(k,f)$.
\end{proof}

\begin{remark}
The proof of metastability result above does not require the projection argument neither sequential weak compactness. Nevertheless, in the presence of the projection point $\proj_{\Fix T}(0)$, following the arguments culminating in \eqref{e:quasi_convergence} one can say that Theorem~\ref{t:main} corresponds to a quantitative version of Theorem~\ref{t:original}.
\end{remark}

\begin{corollary}\label{c:Cor1}
Let $\alpha \in (0,1]$ and $T:H \to H$ be $\alpha$-averaged. Given $(\beta_n) \subset (0,1]$, $(\lambda_n) \subset (0, \frac{1}{\alpha}]$ and $x_0 \in H$, consider $(x_n)$ generated by \eqref{mKM}. Let $a \in \N \setminus\{0\}$ be such that $\alpha\geq \frac{1}{a}$. Assume that there exist $\ell \in \N\setminus \{0\}$ and monotone functions $b, D,B,L: \N \to \N$ and $h: \N\to \N\setminus \{0\}$ such that the conditions $(Q_{\ref{Q1}})-(Q_{\ref{Q6}})$ hold. Assume that there exists $N \in \N \setminus\{0\}$ such that $N \geq \max\{\norm{x_0-p},\norm{p}\}$, for some $p \in \Fix T$.
Then for all $k \in \N$ and monotone function $f:\N\to \N$
\begin{equation*}
\exists n \leq \mu_1(k,f) \, \forall i,j \in [n,f(n)] \,(\norm{x_i-x_j}\leq \frac{1}{k+1}),
\end{equation*}
 where $\mu_1(k,f):=\mu_1[a,N,\ell,b,h,D,B,L](k,f):=\mu[N,a\ell,b,h,D,B,L](k,f)$, with $\mu$ as in Theorem~\ref{t:main}.
\end{corollary}
\begin{proof}
Since $T$ is $\alpha$-averaged, there exists $T'$ nonexpansive and such that $T=(1-\alpha)\Id+\alpha T'$. It is easy to see that $(x_n)$ is generated by \eqref{mKM}, with starting point $x_0$, using the sequences $(\beta_n),(\alpha\lambda_n) \subset (0,1]$ and $T'$. Since $\Fix T = \Fix T'$, by Theorem~\ref{t:main}, we just need to check that the conditions $(Q_{\ref{Q5}})$ and $(Q_{\ref{Q6}})$ still hold using the sequence $(\alpha\lambda_n)$ instead of the sequence $(\lambda_n)$. The condition $(Q_{\ref{Q5}})$ holds with $\ell:=a\ell$, since for all $n \in \N$ we have $\alpha \lambda_n\geq \frac{\alpha}{\ell}\geq \frac{1}{a\ell}$. We have, for all $k,n \in \N$ that since $\alpha \in (0,1]$ $$\sum_{i=L(k)+1}^{L(k)+n}|\alpha\lambda_{i}-\alpha\lambda_{i-1}|\leq \sum_{i=L(k)+1}^{L(k)+n}|\lambda_{i}-\lambda_{i-1}|\leq\frac{1}{k+1}.$$ Hence $(Q_{\ref{Q6}})$ is also satisfied for the sequence $(\alpha\lambda_n)$ with the function $L$.
\end{proof}

\begin{corollary}\label{c:Cor2}
Let $T_1: H \rightrightarrows H$ be maximal monotone  and $T_2:H \to H$ be $\delta$-cocoercive, for some $\delta >0$. Let $\gamma \in (0,2\delta]$. Given $(\beta_n) \subset (0,1]$, $(\lambda_n) \subset (0, \frac{4\delta-\gamma}{2\delta}]$ and $x_0 \in H$, consider $(x_n)$ generated, for all $n \in \N$, by 
\begin{equation}\label{TFB}\tag{$\mathsf{T}\textup{-}\mathsf{FB}$}
x_{n+1}=(1-\lambda_n)\beta_nx_n+\lambda_nJ_{\gamma T_1}(\beta_nx_n-\gamma T_2(\beta_nx_n))
\end{equation}
 Assume that there exist $\ell \in \N\setminus \{0\}$ and monotone functions $b, D,B,L: \N \to \N$ and $h: \N\to \N\setminus \{0\}$ such that the conditions $(Q_{\ref{Q1}})-(Q_{\ref{Q6}})$ hold. Assume that there exists $N \in \N \setminus\{0\}$ such that $N \geq \max\{\norm{x_0-p},\norm{p}\}$, for some $p \in zer (T_1+T_2)$.
Then for all $k \in \N$ and monotone function $f:\N\to \N$
\begin{equation*}
\exists n \leq \mu_2(k,f) \, \forall i,j \in [n,f(n)] (\norm{x_i-x_j}\leq \frac{1}{k+1}),
\end{equation*}
 where $\mu_2(k,f):=\mu_2[N,\ell,b,h,D,B,L](k,f):=\mu[N,2\ell,b,h,D,B,L](k,f)$, with $\mu$ as in Theorem~\ref{t:main}.
\end{corollary}

\begin{proof}
It is straightforward to see that the iteration $(x_n)$ is generated by \eqref{mKM} with $T:=J_{\gamma T_1} \circ (\Id -\gamma T_2)$. The resolvent function $J_{\gamma T_1}$ is firmly nonexpansive, i.e.\ $\frac{1}{2}$-averaged. Furthermore, we have $\frac{\gamma}{2\delta} \in (0,1]$ and, using \cite[Proposition~4.33]{BC(11)} it follows that $(\Id -\gamma T_2)$ is $\frac{\gamma}{2\delta}$-averaged. If $\gamma < 2\delta$, then using \cite[Theorem~3(b)]{OY(02)} it follows that $T$ is $\frac{2\delta}{4\delta-\gamma}$-averaged. If $\gamma = 2\delta$, then $T$ is nonexpansive and therefore also $\frac{2\delta}{4\delta-\gamma}$-averaged.
Since $\Fix T=zer(T_1+T_2)$ \cite[Proposition~25.1(iv)]{BC(11)} we have $N \geq \max\{\norm{x_0-p},\norm{p}\}$, for some $p \in \Fix T$. Noting that $\frac{2\delta}{4\delta-\gamma} \in [\frac{1}{2},1]$, we may apply Corollary~\ref{c:Cor1} with $a=2$ and the result follows.
\end{proof}

\begin{remark}
If in Corollary~\ref{c:Cor2} the resolvent function $J_{\gamma T_1}$ is replaced by an arbitrary firmly nonexpansive mapping $T_1$, and $p$ is some point in $\Fix (T_1 \circ (\Id -\gamma T_2))$, then the result holds also for $\gamma=0$, by Corollary~\ref{c:Cor1}. In such case, the sequence $(\lambda_n)$ is allowed to vary in the interval $(0,2]$.
\end{remark}

\begin{corollary}
Let $T_1,T_2: H \rightrightarrows H$ be two maximal monotone operators and $\gamma  >0$. Given $(\beta_n) \subset (0,1]$, $(\lambda_n) \subset (0, 2]$ and $x_0 \in H$, consider $(x_n)$ generated, for all $n \in \N$, by
\begin{equation}\label{TDR}\tag{$\mathsf{T}\textup{-}\mathsf{DR}$}
\begin{cases}
y_n=J_{\gamma T_2}(\beta_nx_n)\\
z_n=J_{\gamma T_1}(2y_n-\beta_nx_n)\\
x_{n+1}=\beta_nx_n+\lambda_n(z_n-y_n) 
\end{cases}
\end{equation}
Assume that there exist $\ell \in \N\setminus \{0\}$ and monotone functions $b, D,B,L: \N \to \N$ and $h: \N\to \N\setminus \{0\}$ such that the conditions $(Q_{\ref{Q1}})-(Q_{\ref{Q6}})$ hold. Assume that there exists $N \in \N \setminus\{0\}$ such that $N \geq \max\{\norm{x_0-p},\norm{p}\}$, for some $p \in \Fix (R_{\gamma T_1} \circ R_{\gamma T_2})$.
Then, for all $k \in \N$ and monotone function $f:\N\to \N$
\begin{enumerate}[$(i)$]
\item $\exists n \leq \mu_3(k,f) \, \forall i,j \in [n,f(n)] \,(\norm{x_i-x_j}\leq \frac{1}{k+1})$,
\item $\exists n \leq \mu_4(k,f) \, \forall i,j \in [n,f(n)] \,(\norm{y_i-y_j}\leq \frac{1}{k+1})$,
\item $\exists n \leq \mu_5(k,f) \, \forall i,j \in [n,f(n)] \,(\norm{z_i-z_j}\leq \frac{1}{k+1})$,
\end{enumerate}
 where $\mu_3(k,f):=\mu_3[N,\ell,b,h,D,B,L](k,f):=\mu[N,2\ell,b,h,D,B,L](k,f)$, $\mu_4(k,f):=\max\{\mu_3(2k+1,g_1),b(8N(k+1)-1)\}$ and\\ $\mu_5(k,f):=\max\{\mu_4(3k+2,g_2),\nu_1(6\ell (k+1)-1), b(12 \ell N(k+1)-1)\}$,\\ with $\mu$ as in Theorem~\ref{t:main},  and $g_1(m):=f(\max\{m,b(8N(k+1)-1)\})$ and $g_2(m):=f(\max\{m,\nu_1(6\ell (k+1)-1), b(12 \ell N(k+1)-1)\})$.
\end{corollary}

\begin{proof}
It is easy to see that the iteration $(x_n)$ is generated by \eqref{mKM} using the sequences $(\beta_n)$ and $(\frac{\lambda_n}{2})$, and with $T:=R_{\gamma T_1} \circ R_{\gamma T_2}$, which is nonexpansive. 
Hence, applying Theorem~\ref{t:main} we obtain Part (i).

By Part (i), there exists $n_0 \leq \mu_3(2k+1,g_1)$ such that $\norm{x_i-x_j}\leq \frac{1}{2(k+1)}$, for all $i,j \in [n_0, g_1(n_0)]$. Define $n:= \max\{n_0, b(8N(k+1)-1)\}\,(\leq \mu_4(k,f))$. Since $[n,f(n)] \subseteq [n_0,g_1(n_0)]$, we have for $i,j \in [n,f(n)]$  
\begin{equation*}
\begin{split}
\norm{y_i-y_j}&\leq \norm{\beta_ix_i-\beta_j x_j} \leq \norm{x_i-x_j}+|\beta_i-\beta_j|\norm{x_j}\\
& \leq \frac{1}{2(k+1)}+ \left(\frac{1}{8N(k+1)}+\frac{1}{8N(k+1)}\right)2N= \frac{1}{k+1},
\end{split}
\end{equation*}
which shows Part (ii).

Define $n_1:=\max\{\nu_1(6\ell (k+1)-1), b(12 \ell N(k+1)-1)\}$, where $\nu_1$ is as in Lemma~\ref{l:asymptotic_regularity}. 
By the definition of $x_{n+1}$, for $i \geq n_1$
\begin{equation*}
\begin{split}
\norm{z_i-y_i} & \leq \ell \norm{x_{i+1}-\beta_ix_i} \leq \ell (\norm{x_{i+1} -x_i} + 2N(1-\beta_i))\\
& \leq \frac{\ell}{6 \ell (k+1)} + \frac{2 \ell N}{12 \ell N (k+1)}= \frac{1}{3(k+1)}
\end{split}
\end{equation*}
By Part (ii), there exists $n_2 \leq \mu_4(3k+2,g_2)$ such that $\norm{y_i-y_j}\leq \frac{1}{3(k+1)}$, for all $i,j \in [n_2, g_2(n_2)]$.
Define $n:=\max\{n_1,n_2\} \, (\leq \mu_5(k,f))$.  Since $[n,f(n)] \subseteq [n_2,g_2(n_2)]$, we have for $i,j \in [n,f(n)]$
\begin{equation*}
\norm{z_i-z_j}\leq \norm{z_i-y_i}+\norm{y_i-y_j}+\norm{y_j-z_j}\leq \frac{1}{k+1},
\end{equation*}
which concludes the result.
\end{proof}



\bibliography{References}{}
\bibliographystyle{plain}




\end{document}